\documentclass[12pt]{amsart}
\usepackage{times}
\usepackage{amsfonts}
\usepackage{amsthm}
\usepackage{amsmath}
\usepackage{epic}

\usepackage{amssymb,latexsym,amsmath,amsthm,amsfonts,graphics}

\advance \topmargin by -\headheight
\advance \topmargin by -\headsep
\evensidemargin \oddsidemargin
\def\D{\{0,1\}}
\newcommand{\nats}{{\mathbb N}}
\newcommand{\ints}{{\mathbb Z}}

\def\N{\mathbb{N}}
\def\Z{\mathbb{Z}}
\def\A{\mathbb{A}}
\def\R{\mathbb{R}}
\def\E{\mathbb{E}}
\def\NN{\mathcal{N}}

\def\1{\mathbf{1}}
\newcommand{\abp}{p^{\mbox{ab}}}
\newcommand{\mabp}{p^{*\mbox{ab}}}

\newtheorem{thm}{Theorem}

\newtheorem{theorem}{Theorem}[section]
\newtheorem{lemma}[theorem]{Lemma}
\newtheorem{corollary}[theorem]{Corollary}
\newtheorem{proposition}[theorem]{Proposition}

\newtheorem{remark}[theorem]{Remark}

\newtheorem{definition}[theorem]{Definition}
\theoremstyle{definition}

\begin{document}
\parindent=1em

\title{Abelian maximal pattern complexity of words}

\author[T. Kamae]{Teturo Kamae}
\address{Advanced Mathematical Institute, Osaka City University, Osaka, 558-8585 Japan}
\email{kamae@apost.plala.or.jp}

\author[S. Widmer]{Steven Widmer}
\address{Department of Mathematics
General Academics Building 435
1155 Union Circle \#311430
Denton, TX 76203-5017, USA}
\email{s.widmer1@gmail.com}

\author[L.Q. Zamboni]{Luca Q. Zamboni}
\address{Institut Camille Jordan\\
Universit\'e Claude Bernard Lyon 1\\
43 boulevard du 11 novembre 1918\\
F69622 Villeurbanne Cedex, France \indent and \indent Department of Mathematics and
                        Turku Centre for Computer Science,
                       University of Turku,
                        20014 Turku, Finland}
\email{lupastis@gmail.com}
\date{}
\thanks{The third author is supported by a FiDiPro grant from the Academy of Finland.}

\begin{abstract}In this paper we study the maximal pattern complexity of infinite words up to Abelian equivalence. We compute a lower bound for the Abelian maximal pattern complexity of infinite words which are both recurrent and aperiodic by projection. We show that in the case of binary words, the bound is actually achieved and gives a characterization of recurrent aperiodic words.
\end{abstract}

\maketitle

\section{Introduction}

Let $\A$ be a finite non-empty set. We denote by $\A^*$, $\A^\nats$ and $\A^\ints$ respectively the set of finite words, the set of (right) infinite words, and the set of bi-infinite words over the alphabet $\A$. 
Given an infinite word
$\alpha= \alpha _0\alpha_1 \alpha_2\ldots \in \A^\nats$ with $\alpha_i\in \A$,
we denote by ${\mathcal F}_{\alpha}(n)$ the set of all {\it factors} of $\alpha$ of length $n,$ that is, the set of all finite words of the form $\alpha_{i}\alpha_{i+1}\cdots \alpha_{i+n-1}$ with $i\geq 0.$
We set \[p_{\alpha}(n)=\#\,({\mathcal F}_{\alpha}(n)).\] The function $p_{\alpha}:\nats \rightarrow \nats$ is called the {\it factor complexity function} of $\alpha.$

We recall that two words $u$ and $v$ in $\A^*$ are said to be {\it Abelian equivalent,} denoted $u\sim_{\mbox{ab}} v,$  if and only if $|u|_a=|v|_a$ for all $a\in \A,$ where $|u|_a$  denotes the number of occurrences of the letter $a$ in $u.$   It is readily verified that $\sim_{\mbox{ab}}$ defines an equivalence relation on $\A^*.$ 
We define \[{\mathcal F}^{\mbox{ab}}_{\alpha}(n)={\mathcal F}_{\alpha}(n)/\sim_{\mbox{ab}}\] and set \[\abp_{\alpha}(n)=\#\,({\mathcal 
F}^{\mbox{ab}}_{\alpha}(n)).\]
The function
$\abp_{\alpha} :\nats \rightarrow \nats$ which counts the number of pairwise non Abelian equivalent factors of $\alpha$ of length $n$ is called the {\it Abelian complexity} of $\alpha$ (see \cite{RSZ}). 

There are a number of similarities  between the usual factor complexity of an infinite word and its Abelian counterpart.  For instance, both may be used to characterize periodic bi-infinite words (see \cite{MorHed1940} and \cite{CovHed}). 
A word $\alpha$  is \emph{periodic} if there exists a positive integer $p$ such that 
$\alpha_{i+p} = \alpha_i$ for all indices $i$, and it is \emph{ultimately periodic} if $\alpha_{i+p} = \alpha_i$ for all sufficiently large $i$.
An infinite word is \emph{aperiodic} if it is not ultimately periodic.
The factor complexity function also provides a characterization of ultimately periodic words. On the other hand,  Abelian complexity does not yield such a characterization. Indeed, both Sturmian words and the ultimately periodic word  $01^\infty = 0111\cdots$  have the same, constant $2$, Abelian complexity.

\noindent As another example, both complexity functions give a characterization of  \textit{Sturmian} words amongst all aperiodic words:

\begin{thm}\label{SturmianChar}
Let $\alpha$ be an aperiodic infinite word over the alphabet $\{0,1\}$. The following conditions are equivalent:
\begin{itemize}
\item The word $\alpha$ is balanced, that is, {\it Sturmian}.
\item (M. Morse, G.A. Hedlund, \cite{MorHed1940}). The word $\alpha$ satisfies $p_{\alpha}(n_0) = n+1$ for all $n \geq 0$. 
\item (E.M. Coven, G.A. Hedlund, \cite{CovHed}). The word $\alpha$ satisfies $p_{\alpha}^{\mbox{\rm ab}}(n)=2$ for all $n\geq 1$. 
\end{itemize}
\end{thm}

In \cite{KZ}, the first and third authors introduced a different notion of the complexity of an infinite word called the maximal pattern complexity: 

For each positive integer $k,$ let $\Sigma_k(\N)$ denote the set of all $k$-element subsets of $\N.$ 
An element $S=\{s_1<s_2<\cdots<s_k\}\in \Sigma_k(\N)$  will be called a $k$-{\it pattern}. 
We put \[\alpha[S]:=\alpha(s_1)\alpha(s_2)\cdots\alpha(s_k)\in\A^k.\] 
For each $n\in \N,$ the word $\alpha[n+S]$ is called a 
$S$-{\it factor} of $\alpha$, where  
$n+S:=\{n+s_1,n+s_2,\cdots,n+s_k\}$. We denote by ${\mathcal F}_\alpha(S)$ the set of all $S$-factors of $\alpha.$ 
We define the {\it pattern complexity} $p_\alpha(S)$ by
\[p_\alpha(S)=\#\,{\mathcal F}_\alpha(S)\]
and the {\it maximal pattern complexity} $p_\alpha^*(k)$ by
\[p_\alpha^*(k)=\sup_{S\in \Sigma_k(\N)} p_\alpha(S).\] 

In \cite{KZ} the authors show that maximal pattern complexity also gives a characterization of ultimately periodic words :

\begin{thm}\label{patternsturm} Let $\alpha \in \A^\nats.$ Then the following are equivalent
\begin{enumerate}
\item $ \alpha$ is eventually periodic
\item $p_\alpha^*(k)$ is uniformly bounded in $k$
\item $p_\alpha^*(k)<2k$ for some positive integer $k.$
\end{enumerate}
\end{thm}

In other words, $\alpha$ is aperiodic if and only if $p_\alpha^*(k)\geq 2k$ for each positive integer $k.$ We say $\alpha \in \A^\nats$ is {\it pattern Sturmian} if $p_\alpha^*(k)=2k$ for each positive integer $k.$
Two types of recurrent pattern Sturmian words are known:  rotation words (see below) and a family of `simple' Toeplitz words (see \cite{KZ}). Unfortunately, to date  there is no known
classification of  recurrent pattern Sturmian words (as in the case of Theorem~\ref{SturmianChar}).

The connection between items (1) and (3) in Theorem~\ref{patternsturm} was generalized by the first author and R. Hui in \cite{L}. We say $\alpha \in \A^\nats$  is  {\it periodic by projection} 
if there exists a set
$\emptyset\neq B\subsetneqq\A$, such that 
\[
\1_B(\alpha):=\1_B(\alpha(0))\1_B(\alpha(1))\1_B(\alpha(2))\cdots\in\D^\N.
\] is eventually periodic (where $\1_B$ denotes the characteristic function of  $B).$ We say $\alpha$ is {\it aperiodic by projection} if $\alpha$ is not periodic by projection. Then:

 \begin{thm} Let $\#\A=r \geq 2,$ and $\alpha \in \A^\nats$ be aperiodic by projection. Then   $p_\alpha^*(k)\geq r k$ for each
 positive integer $k.$  
 \end{thm}
 
 \noindent In other words, low pattern complexity (relative to the size of the alphabet) implies periodic by projection.  Notice that if $\#\A=2,$ then $\alpha$ is periodic by projection if and only if
 $\alpha$ is eventually periodic.\\
 
 In this paper we introduce and study an Abelian analogue of maximal pattern complexity:
 Given a $k$-pattern $S\in \Sigma_k(\N),$ we define  \[{\mathcal F}^{\mbox{ab}}_{\alpha}(S)={\mathcal F}_{\alpha}(S)/\sim_{\mbox{ab}}\] and the associated  {\it Abelian pattern complexity}
 \[p_\alpha ^{\mbox{ab}}(S)=\#\,{\mathcal F}^{\mbox{ab}}_{\alpha}(S)\]
 which counts the number of pairwise non Abelian equivalent $S$-factors of $\alpha.$ We define the
 {\it Abelian maximal pattern complexity} 
 \[\mabp_{\alpha}(k)=\sup_{S\in \Sigma_k(\N)}p_\alpha ^{\mbox{ab}}(S).\]

 It is clear that for each positive integer $k$ and for each pattern $S\in \Sigma_k(\N)$ we have 
\[p_\alpha^{\mbox{ab}}(S)\leq p_\alpha(S)\,\,\mbox{and }\,\,\mabp_{\alpha} (k)\leq p_\alpha^*(k).\] 
In this paper we show :

\begin{thm}\label{main}
Let $\#\A=r\geq 2$ and $\alpha\in\A^\N$ be recurrent and aperiodic by projection.  
Then for each positive integer $k$ we have
\[
p_\alpha^{*\mbox{\rm ab}}(k)\geq (r-1)k+1
\]
In case $r=2$ equality always holds. Moreover for $k=2$ and general $r,$ 
there exists $\alpha$ satisfying the equality. 
\end{thm}

For example, if $\alpha\in\D^\N$ is a Sturmian word and $S\in \Sigma_k(\N)$ is a $k$-block pattern, i.e., $S=\{0,1,2,\ldots ,k-1\},$  then we have $p_{\alpha}^{\mbox {ab}}(S)=2$ (since $\alpha$ is balanced) while $p_\alpha(S)=k+1.$ Since  $\alpha$ is both recurrent and aperiodic, it follows from the above theorem that the Abelian maximal pattern complexity $\mabp_{\alpha}(k)$ takes the maximum value $k+1$ for each positive integer $k.$ Moreover, all recurrent pattern Sturmian words share this property.

For a rotation word $\alpha\in\A^\nats$ with $r=\#\A\ge 3$, we show that $\mabp_{\alpha}(k)=rk$ for each positive integer $k$ (see Theorem~\ref{rot}). Since $p_\alpha^*(k)=rk$, the abelianization doesn't decrease the complexity in this case. On the other hand, in the proof of Theorem~\ref{main}, we show that $\mabp_{\alpha}(2)=2r-1$ for any Toeplitz word $\alpha\in\A^\N$ with $\#\A=2.$

We define two classes of words with $\A=\{0,1,\cdots,r-1\}$ and $r\ge 2$. 
Let $\theta$ be an irrational number and $c_0<c_1<\cdots<c_{r-1}<c_r$ be real numbers such that $c_r=c_0+1$. Define $\alpha\in\A^k$ by $\alpha(n)=i$ if $n\theta\in[c_i,c_{i+1})\pmod{1}$ for any $i\in\A$ and $n\in\N$. We call such $\alpha$ a {\it rotation word}. 
Let $\Z_2$ be the 2-adic compactification of $\Z$ and $\gamma\in\Z_2$. For $n\in\Z_2$, let $\tau(n)\in\N\cup\{\infty\}$ be the superimum of $k\in\N$ such that $2^k$ devides $n$. Let $B_i~(i\in\A)$ be infinite subsets of $\N\cup\{\infty\}$ such that $B_i\cap B_j=\emptyset$ for any $i,j\in\A$ with $i\ne j$ and $\cup_{i\in\A} B_i=\N\cup\{\infty\}$. Define $\alpha\in\A^\N$ by $\alpha(n)=i$ if $\tau(n-\gamma)\in B_i$ for any $i\in\A$ and $n\in\N.$ 
We call such $\alpha$ a {\it Toeplitz word}.

We do not know whether the inequality in Theorem~\ref{main} is tight when $r\ge 3$ and $k\ge 3$. 

\section{Background \& notation}
Given a finite non-empty set $\A,$ we  endow $\A^\nats$ with the topology generated by the metric
\[d(x, y)=\frac 1{2^n}\,\,\mbox{where} \,\, n=\inf\{k :x_k\neq y_k\}\] 
whenever $x=(x_n)_{n\in \nats}$ and $y=(y_n)_{n\in \nats}$ are two elements of $\A^\nats.$
For $\omega \in \A^\nats,$ let $\overline{O}(\omega)$ denote
the closure  of the orbit $O(\omega):=\{T^n\omega :\,~n\in\N\}$ of $\omega$ with respect 
to the shift $T$ on $\A^\N$, where $(T\omega)(n)=\omega(n+1)~(n\in\N)$.

Given a finite word $u =a_1a_2\ldots a_n$ with $n \geq 1$ and $a_i \in A,$ we denote the length $n$ of $u$ by $|u|.$ For each $a\in A,$ we let $|u|_a$  denote the number of occurrences of the letter $a$ in $u.$

For each $u\in A^*,$ we denote by
$\Psi(u)$ the {\it Parikh vector} or {\it abelianization} of $u,$ that is the vector indexed by $\A$
\[\Psi(u)=(|u|_a)_{a\in \A}.\]
Given $\Xi\subset\A^*$, we set  
\[\Xi^{\mbox{ab}}:=\Xi/ \sim_{\mbox{ab}}\] and  
\[\Psi(\Xi):=\{\Psi(\xi)\,|\,~\xi\in\Xi\}.\]
There is an obvious bijection between the sets $\Xi^{\mbox{ab}}$ and $\Psi(\Xi)$ where one identifies
the Abelian class of an element $u\in \A^*$ with its Parikh vector $\Psi(u).$

Given a nonempty set $\Omega\subset\A^\N,$ $S\in \Sigma_k(\N)$  and an infinite set 
$\NN\subset\N$ we put \[\Omega[S]:=\{\omega[S]\,|\,\omega\in\Omega\}\subset\A^k\] and
\[\Omega[\NN]:=\{\omega[\NN]\,|\,\omega\in\Omega\}\subset\A^\N\] where
$\omega[\NN]\in\A^\N$ is defined by $\omega[\NN](n)=\omega(N_n)~(n\in\N).$

Analogously we can define the maximal pattern complexity of $\Omega$ by
\[p_\Omega^*(k)=\sup_{S\in \Sigma_k(\N)} p_\Omega(S)\] 
where
\[p_\Omega(S)=\#\,\Omega[S]\]
and the Abelian maximal pattern complexity of $\Omega$
 \[\mabp_{\Omega}(k)=\sup_{S\in \Sigma_k(\N)}p_\Omega ^{\mbox{ab}}(S)\]
 where
 
 \[p_\Omega ^{\mbox{ab}}(S)=\#\,\Omega[S]^{\mbox{ab}}.\]

\section{Superstationary sets \& Ramsey's Infinitary Theorem}

\begin{lemma}\label{*} Let $\omega \in \A^\nats$ be a recurrent infinite word.  Then there exists an infinite set $\NN=\{N_0<N_1<N_2<\cdots\}\subset\N$ satisfying the following condition:
\[(*)\,\,\,\,\forall i\geq 0,\,\forall k\geq 0\,\,\,  \omega_{i+N_0}\omega_{i+N_1}\cdots \omega_{i+N_{k-1}}\omega_{i+N_k}^\infty\in \overline{O}(\omega) [\NN]\] 
\end{lemma}

\begin{proof} We show by induction on $k$ that for each $k\geq 0$ there exists natural numbers $N_0<N_1<\cdots <N_k$ such that for each $j\leq k,$ if $u_0u_1\cdots u_j \in \overline{O}(\omega) [\{N_0,N_1,\ldots,N_j\}]$ then
$u_0u_1\cdots u_j^{k-j+1}\in \overline{O}(\omega) [\{N_0,N_1,\ldots,N_k\}].$
Clearly we can take for $N_0$ any natural number in $\N.$ Next suppose we have chosen natural numbers 
$N_0<N_1<\cdots <N_k$ with the required property. Fix a positive integer $L$ such that if $u\in \overline{O}(\omega) [\{N_0,N_1,\ldots,N_k\}],$ then there exists $i\leq L-N_k$  with $u=\omega_{i+N_0}\omega_{i+N_1}\cdots\omega_{i+N_k}.$  Since $\omega$ is recurrent, there exists a positive integer$N_{k+1}>N_k$ such that $\omega_i=\omega_{i+N_{k+1}}$ for each $i\leq L.$
We now verify that $N_0<N_1<\cdots <N_{k+1}$ satisfies the required property. So assume $j\leq k+1$ and
$u_0u_1\cdots u_j \in \overline{O}(\omega) [\{N_0,N_1,\ldots,N_j\}].$ We must show that
$u_0u_1\cdots u_j^{k+1-j+1}\in \overline{O}(\omega) [\{N_0,N_1,\ldots,N_{k+1}\}].$ This is clear in case  $j=k+1,$ thus we can assume $j\leq k.$ Then by induction hypothesis we have that
$u= u_0u_1\cdots u_j^{k-j+1}\in \overline{O}(\omega) [\{N_0,N_1,\ldots,N_k\}].$ Fix $i\leq L-N_k$ such that $u=\omega_{i+N_0}\omega_{i+N_1}\cdots\omega_{i+N_k}.$ Then

\[\begin{array}{ll}
u_0u_1\cdots u_j^{k+1-j+1}&=u_0u_1\cdots u_j^{k-j+1}u_j\\{}&=\omega_{i+N_0}\omega_{i+N_1}\cdots\omega_{i+N_k}\omega_{i+N_k}\\{}&=\omega_{i+N_0}\omega_{i+N_1}\cdots\omega_{i+N_k}\omega_{i+N_{k+1}}.
\end{array}\]
Hence $u_0u_1\cdots u_j^{k+1-j+1}\in \overline{O}(\omega) [\{N_0,N_1,\ldots,N_{k+1}\}]$
as required.
\end{proof}

\noindent It is readily verified that:

\begin{lemma}\label{sub*} Let $\NN=\{N_0<N_1<N_2<\cdots\}\subset\N$ be an infinite set satisfying the condition (*) above, and let $\NN'$ be any infinite subset of $\NN.$ Then $\NN'$ also satisfies (*).
\end{lemma}

\begin{proposition}\label{ram} Let $\Omega \subset \A^\nats$ be non-empty and let $\NN=\{N_0<N_1<N_2<\cdots\}\subset\N$ be an infinite set. Then for every positive integer $k,$ there exists an infinite subset $\NN'$ of $\NN$ (depending on $k)$ such that for any two finite subsets $P$ and $Q$ of $\NN'$ with $1\leq |P|=|Q|\leq k,$  we have\[\Omega[P]=\Omega[Q].\]
\end{proposition}

\begin{proof} We will recursively construct a sequence of nested infinite patterns
\[\NN'=\NN_k\subset \cdots \subset \NN_2 \subset \NN_1= \NN\]
such that for each $1\leq i \leq k$ we have
\[\Omega[P]=\Omega[Q]\]
for all finite subsets $P$ and $Q$ of $\NN_i$ with $1\leq |P|=|Q|\leq i.$ 

We begin with $\NN_2.$ Given two finite sub-patterns $P$ and $Q$ of $\NN_1$ with $|P|=|Q|=2,$ we write
\[P\sim _2 Q \iff \Omega[P]=\Omega[Q]. \] 
Then $\sim_2$ defines an equivalence relation on the set of all sub-patterns of $\NN_1$ of size $2,$ and hence naturally defines a finite coloring on the set of all size $2$ sub-patterns of $\NN_1,$ or equivalently on the set of all $2$-element subsets of the natural numbers $\nats,$ where two patterns $P$ and $Q$ are monochromatic if and only if $P\sim_2Q.$
We now recall the following well known theorem of Ramsey:

\begin{thm}[\cite{R},~Ramsey]] Let $k$ be a positive integer. Then given any finite coloring of  the set of all $k$-element subset of $\nats,$ there exists an infinite set $\mathcal A\subset \nats$ such that any two $k$-element subsets of $\mathcal A$ are monochromatic.
\end{thm}

Thus applying the above theorem we deduce that there exists an infinite
pattern $\NN_2\subset \NN_1$ such that any two sub-patterns $P$ and $Q$ of $\NN_2$ of size $2$ are $\sim_2$ equivalent.

Having constructed $\NN_k\subset \NN_{k-1}\subset \cdots \subset \NN_2\subset \NN_1=\NN$ with the required properties, we next construct
$\NN_{k+1}$ as follows:  Given any two sub-patterns $P$ and $Q$ of $\NN_k$ of size $k+1,$ we write
\[P\sim_{k+1}Q \iff \Omega[P]=\Omega[Q].\]
Again this defines a finite coloring of the set of all size $k+1$ sub-patterns of
$\NN_k,$ or equivalently on the set of all $(k+1)-$element subsets of $\nats.$  Hence by Ramsey's theorem, we deduce that there exists an infinite
pattern $\NN_{k+1}\subset \NN_k$ such that
any two sub-patterns of $\NN_{k+1}$ of size $k+1$ are monochromatic, i.e., $\sim_{k+1}$ equivalent. Moreover, since $\NN_{k+1}\subset \NN_k,$ it follows that any two sub-patterns $P$ and $Q$  of $\NN_{k+1}$
of size $1\leq |P|=|Q|\leq k$ are $\sim _{|P|}$ equivalent.

\end{proof}

\begin{definition} Let $k\ge 2$. A nonempty set $\Omega\subset\A^\N$ is called a $k$-{\it superstationary} set if \[\Omega[S]=\Omega[S']\]  for any $S$ and $S'\in \Sigma_k(\N)$ (see \cite{G}). 
\end{definition}

As an immediate consequence of Proposition~\ref{ram} we have

\begin{corollary}\label{ss} Let $\Omega \subset \A^\nats$ be non-empty and let $\NN\subset\N$ be an infinite set. Then for every positive integer $k,$ there exists an infinite subset $\NN'$ such that $\Omega[\NN']$ is $k$-superstationary.
\end{corollary}

\begin{lemma}\label{conn} Let $\alpha \in \A^\nats$ be aperiodic by projection and let $\NN\subset \nats$ be any infinite set.  Put $\Omega:=\overline{O}(\alpha)[\NN].$ Then for any $\{i<j\}\subset\N$, the directed graph $(\A,E_{i,j})$ is 
strongly connected, where 
\[E_{i,j}=\{(\omega(i),\omega(j))\in\A\times\A;~\omega\in\Omega,\,\,\omega(i)\neq \omega(j)\}.\]
\end{lemma}

\begin{proof}  Fix $\{i<j\}\subset\N$ and $\NN=\{N_0<N_1<\cdots\}$. For any $l=0,1,\cdots,N_j-N_i-1$, let $A_l$ be the set of $a\in\A$ such that $\alpha(n)=a$ holds for infinitely many $n\in\N$ with $n\equiv l\pmod{N_j-N_i}$. For any $a,b\in\A$, if $\{a,b\}\in A_l$ for some $l\in\{0,1,\cdots,N_j-N_i-1\}$, then $a,b$ are two way connected in the graph $(\A,E_{i,j})$. Hence for $a,b\in\A$, $a,b$ are two way connected in the graph $(\A,E_{i,j})$ if there exist $a_0,a_1,\cdots,a_k\in\A$ and $l_1,\cdots,l_k\in\{0,1,\cdots,N_j-N_i-1\}$ such that
(i)~$a_0=a$, $a_k=b$, and
(ii)~$\{a_{i-1},a_i\}\in A_{l_i}$ for any $i=1,\cdots,k$.

Suppose to the contrary that there exist $a,b\in\A$ such that $a$ and $b$ are not two way connected in the graph $(\A,E_{i,j})$. Let $A$ be the set of $a'\in\A$ such that $a,a'$ are two way connected in the graph $(\A,E_{i,j})$. Then, we have $\emptyset\ne A{\subset\atop{\ne}}\A$. Moreover, there exists $S$ with $\emptyset\ne S\subset\{0,1,\cdots,N_j-N_i-1\}$ such that $A_l\subset A$ for any $l\in S$ and $A_l\cap A=\emptyset$ for any $l\in \{0,1,\cdots,N_j-N_i-1\}\setminus S$. Therefore,
$$
1_A(\alpha(0))1_A(\alpha(1))1_A(\alpha(2))\cdots
$$ 
is periodic with period $N_j-N_i$, which contradicts our assumption that $\alpha$ is aperiodic by projection. Thus, the graph is strongly connected.

\end{proof}

\noindent Combining lemmas~\ref{*}, \ref{sub*} and \ref{conn} with Proposition~\ref{ram} we obtain:

\begin{proposition}\label{key}
Let $\omega\in\A^\N$ be recurrent and aperiodic by projection and $k\ge2$. Then there exists an infinite set $\NN\subset\N$ such that $\Omega:=\overline{O}(\alpha)[\NN]$ is a $k$-superstationary set and
\begin{enumerate}
\item For any $\omega\in\Omega$ and $i\in\N$, 
$$\omega(0)\omega(1)\cdots\omega(i-1)\omega(i)^\infty
\in\Omega$$
\item  For any $\{i<j\}\subset\N$, the directed graph $(\A,E_{i,j})$ is 
strongly connected.
\end{enumerate}

\end{proposition}


\section{Main results}

\begin{proof}[Proof of Theorem~\ref{main}]
Fix a positive integer $k.$ 
By Proposition~\ref{key}, there exists an infinite set $\NN\subset\N$ such that
$\Omega=\overline{O}(\alpha)[\NN]\subset\A^\N$ is $k+1$-superstationary and satisfies conditions $(1)$ and $(2)$ of Proposition~\ref{key}. Since $\mabp_{\alpha}(k)\geq \mabp_{\Omega}(k),$ 
 it is sufficient to prove that $\#\,{\Omega_k}^{\mbox{ab}}\geq (r-1)k+1$, where $\Omega_k:=\Omega[\{0,1,\cdots,k-1\}]$. 

Let $(\A,E_{0,1})$ be the strongly directed graph where 
$$E_{0,1}=\{(\omega(0),\omega(1))\in\A\times\A\mbox{ with }\omega(0)\ne\omega(1);~\omega\in\Omega\}.$$
Then there exists a sequence $a_0a_1\cdots a_l$ of elements in $\A$ containing all elements in $\A$ such that $(a_i,a_{i+1})\in E_{0,1}~(i=0,1,\cdots,l-1)$. 

Define a non-directed graph $(\A,F)$ by 
$$F=\{\{a,b\}\subset\A\mbox{ with }a\ne b~\mbox{and either }a^kb^\infty\in\Omega
\mbox{ or }b^ka^\infty\in\Omega\}. $$

Since $\Omega$ is $k+1$-superstationary, for any $i=0,1,\cdots,l-1$, 
there exists $\omega\in\Omega$ such that $\omega[\{kr,kr+1\}]=a_ia_{i+1}$. 
Hence, by $(1)$ of Proposition~\ref{key}, there exists 
$\xi\in\A^{kr}$ such that $\xi a_ia_{i+1}^\infty\in\Omega$ 
and $\xi a_i^\infty\in\Omega$. Then, there exists $b\in\A$ occurring 
in $\xi$ at least $k$ times. Since $\Omega$ is $k+1$-superstationary, this implies that $b^ka_i$ and $b^ka_{i+1}$ are in $\Omega[\{0,1,\cdots,k\}]$. Therefore, $b^ka_i^\infty\in\Omega$ and $b^ka_{i+1}^\infty\in\Omega$ by $(1)$ of Proposition~\ref{key}. Hence, we have two cases according to whether $b\in\{a_i,a_{i+1}\}$ or not. \vspace{0.5em}\\
{\bf{Case 1:}}~$b\in\{a_i,a_{i+1}\}$. In this case, we have $\{a_i,a_{i+1}\}\in F$. \vspace{0.5em}\\
{\bf{Case 2:}}~$b\notin\{a_i,a_{i+1}\}$. In this case, we have $2$ edges $\{b,a_i\}$ and $\{b,a_{i+1}\}$ in $F$, by which $a_i$ and $a_{i+1}$ are connected. \vspace{-0.8em}\\

Thus, we have a connected graph $(\A,F)$. This implies there are 
at least $r-1$ edges. If $\{a,b\}\in F$, then either $a^kb^\infty\in\Omega$ 
or $b^ka^\infty\in\Omega$. Since $\Omega$ is $k$-superstationary, either $a^hb^{k-h}\in\Omega_k~(h=0,1,\cdots,k)$ or 
$b^ha^{k-h}\in\Omega_k~(h=0,1,\cdots,k)$. Any case, there are $k+1$ 
elements in ${\Omega_k}^{\mbox{ab}}$ consisting only of $a$ and $b$. 

Since $\#F\ge r-1$, there are at least $(r-1)(k+1)-(r-2)=(r-1)k+1$ 
elements in ${\Omega_k}^{\mbox{ab}}$ consisting only of 2 elements, 
where we subtract $r-2$ since the number of overlapping counted for 
constant words is $2(r-1)-r=r-2$. 

Thus, $\#\,\Omega_k^{\mbox{ab}}\ge (r-1)k+1$. \vspace{-0.8em}\\

If $\#\A=2$, then $\mabp_{\alpha}(k)
\le k+1~~(k=1,2,\cdots)$ for any $\alpha\in\A^\N$, since the number of 
vectors $(|\xi|_0,|\xi|_1)$ over all $\xi\in\D^k$ is $k+1$. 

Let $\A=\{0,1,\cdots,r-1\}$ with $r\ge 3$. For $n\ge 1$, 
let $\tau(n)$ be the maximum $\tau\in\N$ such that $2^\tau$ is a factor of $n$. Define $\alpha\in\A^\N$ by $\alpha(n)=\tau(n+1)\pmod{r}$. Then $\alpha$ is one of the Toeplitz words defined in Introduction. It is clearly recurrent and aperiodic by projection. 

Take any $2$-pattern $S=\{s<t\}\subset\N$. Let $d=\tau(t-s)$. Then there 
exists $u\in\N$ with $0\le u<2^d$ such that either 
$\tau(s-u)=d,~ \tau(t-u)>d$ or $\tau(s-u)>d,~ \tau(t-u)=d$. Assume without loss of generality that the latter holds. Let $c\in\A$ be such that $c\equiv d\pmod{r}$ and denote by $\E_a\in\R^\A$ the unit vector at $a\in\A$. There are 3 cases for $n\in\Z$. 
\vspace{0.5em}\\
{\bf{Case 1:}}~$\tau(n+u+1)>d.$~~In this case, $\tau(n+t+1)=d$ holds. 
Hence, $\Psi(\alpha[n+S])=\E_a+\E_c$ for some $a\in\A$.
\vspace{0.5em}\\
{\bf{Case 2:}}~~$\tau(n+u+1)=d.$~~In this case, $\tau(n+s+1)=d$ holds. 
Hence, $\Psi(\alpha[n+S])=\E_a+\E_c$ for some $a\in\A$.
\vspace{0.5em}\\
{\bf{Case 3:}}~~$\tau(n+u+1)<d.$~~In this case, 
$\tau(n+s+1)=\tau(n+t+1)<d$. Hence, $\Psi\alpha[n+S])=2\E_a$ 
for some $a\in\A$. \vspace{-0.8em}\\

Therefore, 
$$\{\Psi(\alpha[n+S])\,:\,n\in\N\}\subset
\{\E_a+\E_c;~a\in\A\}\cup\{2\E_a;~a\in\A\},$$
and hence, $\mabp_{\alpha}(2)\le 2r-1$. Thus, 
$\mabp_{\alpha}(2)=2r-1$ since we already have 
$\mabp_{\alpha}(2)\ge 2r-1$. Note that this proof remains true for any of the general Toeplitz words defined in the Introduction. 
\end{proof}

\begin{remark}{\rm Theorem~\ref{main} is not true without the assumption of 
recurrency. In fact, let $\alpha=10^{3}10^{3^2}10^{3^3}1\cdots\in\D^\N$. 
Then, $\mabp_{\alpha}(3)=3$. To see this, suppose
$\alpha[n+S]=111$  for some $n\in\N,$ and some $3$-pattern $S=\{i<j<k\}.$  Then 
 $j-i=3^b-3^a$ and $k-j=3^c-3^b$ for some positive 
integers $a<b<c$. Moreover, this happens when $n=3^a-i$. 

Suppose $\alpha[m+S]=110$ for some $m.$  Then since there exists positive integers 
$d<e$ such that $m+i=3^d$ and $m+j=3^e,$ we have 
$j-i=(m+j)-(m+i)=3^e-3^d=3^b-3^a$. This implies that $3^e+3^a=3^b+3^d$, 
concluding $e=b$ and $a=d$ by the uniqueness of $3$-adic representation. 
Hence, $m=3^d-i=3^a-i,$  a contradiction. 

If $\alpha[m+S]=101$ for some $m,$  then since there exists positive integers 
$d<e$ such that $m+i=3^d$ and $m+k=3^e,$ we have 
$k-i=(m+k)-(m+i)=3^e-3^d=3^c-3^a$. This implies that $3^e+3^a=3^c+3^d$, 
concluding $e=c$ and $a=d$ by the uniquness of 3-adic representation. 
Hence, $m=3^d-i=3^a-i,$ a contradiction.

Finally, if $\alpha[m+S]=011$ for some $m,$ then since there exists positive integers 
$d<e$ such that $m+j=3^d$ and $m+k=3^e,$ we have 
$k-j=(m+k)-(m+j)=3^e-3^d=3^c-3^b$. This implies that $3^e+3^b=3^c+3^d$, 
concluding $e=c$ and $b=d$ by the uniquness of 3-adic representation. 
Hence, $m=3^d-j=3^b-j=3^a-i,$ again a contradiction.

Thus if $111\in\{\alpha[n+S];~n\in\N\}$ then 
$\{110,101,011\}\cap\{\alpha[n+S];~n\in\N\}=\emptyset$. 
Thus, $\mabp_{\alpha}(3)\leq 3$. Since it is clear that 
$\mabp_{\alpha}(3)\geq 3$, we have $\mabp_{\alpha}(3)=3$. 
}\end{remark}

\begin{remark}{\rm 
We do not know whether there exist $\#\A=r\ge 3$ and $\alpha\in\A^\N$ which is recurrent and aperiodic by projection and  such that
$$
\mabp_{\alpha}(k)=(r-1)k+1~~(k=1,2,\cdots).
$$
Let $\A=\{0,1,\cdots,r-1\}$ and $\Omega=\cup_{i=0}^{r-2}\{i,i+1\}^\N$. Then, it is readily verified that $\mabp _{\Omega}(k)=(r-1)k+1~~(k=1,2,\cdots)$. But $\Omega$ is not equal to $\overline{O}(\alpha)$  for any choice of $\alpha\in\A^\N$. 
}\end{remark}

\begin{thm}\label{rot}
Let $\alpha\in\A^N$ be a rotation word with $\#\A=r$. Then, we have 
$p_\alpha^{*\mbox{\rm ab}}(k)=rk~(k=1,2,\cdots)$. 
\end{thm}

\begin{remark}{\rm 
For a rotation word $\alpha\in\A^N$ with $\#\A=r$, it is known \cite{L} that $p_\alpha^*(k)=rk~(k=1,2,\cdots)$. Hence, Theorem~\ref{rot} shows that the abelianization does not decrease the complexity in the case of rotation words on  more than $2$ letters. 
}\end{remark}

\begin{proof}[Proof of Theorem~\ref{rot}]
Since $\mabp_{\alpha}(k)\le
p_\alpha^*(k)=rk~(k=1,2,\cdots)$, it is sufficient to prove that 
$\mabp_{\alpha}(k)\ge rk~(k=1,2,\cdots)$. Let $\theta$ is an irrational number and $c_0<c_1<\cdots<c_{r-1}<c_r$ be real numbers with $c_r=c_0+1$. Let $\A=\{0,1,\cdots,r-1\}$. We may assume that $\alpha\in\A^\N$ is such that $\alpha(n)=i$ whenever $n\theta\in[c_i,c_{i+1})\pmod{1}$

Fix  $0<\varepsilon<\min_i(c_{i+1}-c_i).$ Set 
$\NN=\{N_0<N_1<\cdots\}\subset\N$ such that 
\[\varepsilon>\{N_0\theta\}>\{N_1\theta\}>\cdots>0\] and 
$\lim_{n\to\infty}\{N_n\theta\}=0.$ Here  
$\{~\}$ denotes the fractional part. Then, it is easy to see that 
$$\{\alpha[n+\NN]\,;\,n\in\N\}=\bigcup_{i=0}^{r-1}\{(i+1)^ni^\infty\,;\,n\in\N\},$$
where we identify $r$ with $0$ as letters. 
Thus, for any $k=1,2,\cdots$, we have 
$$\{\alpha[n+\NN_k]\,;\,n\in\N\}=\bigcup_{i=0}^{r-1}
\{(i+1)^ni^{k-n}\,;\,0\le n\le k\},$$
where $\NN_k=\{N_0<N_1<\cdots<N_{k-1}\}$. There are exactly $rk$ words 
as above. Thus, $\mabp_{\alpha}(k)\ge rk~(k=1,2,\cdots)$, which 
completes the proof.
\end{proof}

\end{document}